\documentclass[11pt, reqno]{amsart}
\usepackage[T1]{fontenc}
\usepackage[utf8]{inputenc}
\usepackage{color}
\usepackage{lmodern}
\usepackage[letterpaper,margin=1in]{geometry}
\usepackage{amsmath,amssymb,amsthm,mathtools}
\usepackage{booktabs,array}
\usepackage{microtype}
\usepackage[hidelinks]{hyperref}
\makeatletter

\def\@tocline#1#2#3#4#5#6#7{%
  \ifnum #1>\c@tocdepth
  \else
    \par\addpenalty\@secpenalty\addvspace{#2}%
    \begingroup
    \hyphenpenalty\@M
    \@ifempty{#4}{%
      \@tempdima\csname r@tocindent\number#1\endcsname\relax
    }{%
      \@tempdima#4\relax
    }%
    \parindent\z@
    \leftskip#3\relax
    \advance\leftskip\@tempdima\relax
    \rightskip\@pnumwidth plus4em
    \parfillskip-\@pnumwidth
    #5\leavevmode\hskip-\@tempdima #6\nobreak
    \ifnum #1>1
      \leaders\hbox{\normalfont\kern.3em.\kern.3em}\hfill
    \else
      \hfill
    \fi
    \hbox to\@pnumwidth{%
      \ifnum #1=1
        \hss\bfseries #7%
      \else
        \@tocpagenum{#7}%
      \fi
    }\par
    \nobreak
    \endgroup
  \fi
}

\def\l@section{\@tocline{1}{7pt}{1pc}{}{\bfseries}}
\def\l@subsection{\@tocline{2}{0pt}{3pc}{}{}}
\def\l@subsubsection{\@tocline{3}{0pt}{5pc}{}{}}

\renewcommand{\tocsection}[3]{%
  \indentlabel{\@ifnotempty{#2}{\ignorespaces#1 #2\quad}}#3}

\makeatother
\hypersetup{pdftitle={Simultaneous Unimodularity of Positive Integer Matrices and Their Entrywise Cubes},pdfsubject={A cofactor-preserving extension in every order at least ten}}
\usepackage{fancyhdr}
\usepackage{listings}
\numberwithin{equation}{section}
\newtheorem{theorem}{Theorem}[section]
\newtheorem{proposition}[theorem]{Proposition}
\newtheorem{lemma}[theorem]{Lemma}
\newtheorem{corollary}[theorem]{Corollary}
\theoremstyle{remark}
\newtheorem{remark}[theorem]{Remark}
\newcommand{\Z}{\mathbb Z}
\newcommand{\R}{\mathbb R}
\newcommand{\Q}{\mathbb Q}
\newcommand{\CC}{\mathbb C}
\newcommand{\bu}{\boldsymbol{u}}
\newcommand{\bv}{\boldsymbol{v}}
\newcommand{\bw}{\boldsymbol{w}}
\newcommand{\by}{\boldsymbol{y}}

\newcommand{\bgamma}{\boldsymbol{\gamma}}

\newcommand{\Fp}{\mathbb F_p}
\newcommand{\T}{\mathbb T}
\newcommand{\bx}{\boldsymbol{x}}
\newcommand{\bc}{\boldsymbol{c}}
\newcommand{\bd}{\boldsymbol{d}}
\newcommand{\br}{\boldsymbol{r}}
\newcommand{\cB}{\mathcal B}
\newcommand{\cM}{\mathfrak M}
\newcommand{\cm}{\mathfrak m}
\newcommand{\fS}{\mathfrak S}
\newcommand{\fJ}{\mathfrak J}
\usepackage{indentfirst}

\DeclareMathOperator{\lcm}{lcm}

\newcommand{\cube}[1]{{#1}^{[3]}}
\allowdisplaybreaks[1]
\title{Simultaneous Unimodularity of Positive Integer Matrices and Their Entrywise Cubes}
\date{September 20, 2026}
\author{Xinyi Li}
\address[Xinyi Li]{Department of Mathematics, University of Illinois at Urbana-Champaign, 1409 West Green 
Street, Urbana, IL 61801, USA}
\email{xinyil25@illinois.edu}

\author{Shuochen Wang}
\address[Shuochen Wang]{Department of Mathematics, University of Illinois at Urbana-Champaign, 1409 West Green 
Street, Urbana, IL 61801, USA}
\email{sw84@illinois.edu}

\author{Zimeng Ye}
\address[Zimeng Ye]{Department of Mathematics, University of Illinois at Urbana-Champaign, 1409 West Green 
Street, Urbana, IL 61801, USA}
\email{zimengy3@illinois.edu}

\author{Alexandru Zaharescu}
\address[Alexandru Zaharescu]{Department of Mathematics, University of Illinois at Urbana-Champaign, 1409 West Green 
Street, Urbana, IL 61801, USA}
\email{zaharesc@illinois.edu}
\begin{document}
\maketitle
\thispagestyle{plain}

\begin{abstract}
We solve an open problem posed by Richard Guy in his book \textit{Unsolved Problems in Number Theory.} We prove that, for every integer $n\geq10$, there exist infinitely many $n\times n$ integer matrices $A$ such that every entry is $\ge2$ and $\det A=\det A^{[3]}=1$, where $A^{[3]}$ denotes the entrywise cube of $A$. The matrices may be chosen compatibly, so they can form the leading principal blocks of a single infinite positive integer array. The order-ten case is reduced to an inhomogeneous system consisting of one linear and one diagonal cubic equation.
The circle method then provides infinitely many positive integral solutions. To pass from order $n$ to order $n+1$, we introduce a bordering construction that preserves the signed cofactor vectors of both the matrix and its entrywise cube. Iterating this construction yields the result in every order $n\geq10$.
\end{abstract}


\tableofcontents

\section{Introduction}\label{sec:intro}

For a matrix $A=(a_{ij})\in M_{n\times n}(\mathbb Z)$, write $\cube{A}=(a_{ij}^3)$.
Here $\cube{A}$ is an entrywise operation of the original matrix $A$, not an ordinary matrix power. We study the existence of such matrix $A$ satisfying
\begin{equation}\label{eq:target}
 \det A=\det\cube{A}=1,\qquad a_{ij}\geq2\quad\text{for all }1\leq i,j\leq n.
\end{equation}
The square-entry problem was discussed in Guy's \textit{Unsolved Problems in Number Theory, F28} \textcolor{red}{\cite{Guy}} and in the work of D\u anescu, V\^aj\^aitu and Zaharescu \textcolor{red}{\cite{DVZ}}. Guy also asked whether the problem extends to cubes. To our knowledge, the only result obtained so far in this direction was established by Choudhry\textcolor{red}{\cite[Theorem 2.2 and p.~157]{Choudhry}}.
Choudhry constructs cubic examples of $3\times3$ matrices with no entries equal to $\pm1$, where one entry in that family is zero, which does not satisfy the restriction in Guy's problem that $a_{ij}\notin\{0,\pm1\}$. 
Our main result gives an affirmative answer for $n \times n$ matrices for all $n\ge10$.

\begin{theorem}\label{thm:main}
For every integer $n\geq10$, there are infinitely many integer matrices of
order $n$ satisfying \eqref{eq:target}. Moreover, there is an infinite array
$(a_{ij})_{i,j\geq1}$ of integers such that
\[
 \det(a_{ij})_{1\leq i,j\leq n}
 =\det(a_{ij}^3)_{1\leq i,j\leq n}=1,
 \qquad a_{ij}\geq 2\qquad(n\geq10).
\]
\end{theorem}

The main difficulty is to pass from a solution of order $n$ to one of order $n+1$ while preserving both determinant conditions and the positivity of all entries. Our construction accomplishes this by preserving the signed cofactor vectors of a positive rectangular lower block when we adjoin one row and one column. Once this lower block is fixed, only the first row of the square matrix varies. This makes the construction stable under iteration and allows the order to be increased one at a time without altering the previously constructed entries.

\section{The order-ten seed and fixed determinant equations}\label{sec:seed}

We use the following explicit $9\times10$ lower block:
\begin{equation}\label{eq:B}
 B=\begin{pmatrix}
 4&3&5&5&4&3&5&4&2&3\\
 4&5&5&2&2&3&4&5&3&5\\
 5&5&3&4&5&4&5&3&5&3\\
 2&3&4&5&2&4&4&3&3&3\\
 4&3&3&2&3&4&2&2&5&2\\
 5&4&5&5&2&5&5&4&3&5\\
 2&4&2&3&5&3&5&4&5&2\\
 4&5&2&4&3&4&2&2&5&4\\
 4&5&2&2&3&4&3&5&4&4
 \end{pmatrix}.
\end{equation}
For a row vector $\bx=(x_1,\ldots,x_{10})$, let
\begin{equation}\label{eq:A}
 A(\bx)=\begin{pmatrix}\bx\\ B\end{pmatrix}.
\end{equation}
Let
\begin{equation}\label{eq:box}
 \br=(4,3,5,5,4,3,5,4,2,3),\qquad
 \cB=\prod_{j=1}^{10}(r_j-1/10,r_j+1/10).
\end{equation}
Note that all coordinates in $\cB$ are positive.

\begin{theorem}\label{thm:seed}
Let $B$, $A(\bx)$ and $\cB$ be given by
\eqref{eq:B}--\eqref{eq:box}. There is a constant $\kappa>0$ such that
\begin{align}\label{eq:maincount}
 &\#\{\bx\in\Z^{10}:\bx/P\in\cB,\
       \det A(\bx)=\det\cube{A(\bx)}=1\}\notag\\
 &\hspace{35mm}=\kappa P^6+O(P^{6-1/100})
 \qquad \text{as }P\longrightarrow\infty.
\end{align}
Consequently there are infinitely many matrices satisfying
\eqref{eq:target} in order ten, with every entry positive. The bottom nine
rows are always the fixed rows of $B$, and every first-row entry tends to
infinity along a suitable sequence of such matrices.
\end{theorem}

The seed theorem is proved in Section~\ref{sec:finish}; its analytic and
local ingredients are given in Sections~\ref{sec:cofactor}--\ref{sec:analytic}.
For a rectangular $(n-1)\times n$ matrix $E$, set
\begin{equation}\label{eq:gammanotation}
 \gamma_j(E)=(-1)^{1+j}\det E_{\widehat j},\qquad
 \bgamma(E)=(\gamma_1(E),\ldots,\gamma_n(E)),
\end{equation}
where $E_{\widehat j}$ means that column $j$ has been deleted from matrix $E$.
Write $\bc=\bgamma(B)$ and $\bd=\bgamma(\cube B)$ for the two seed vectors.
Their entries are displayed in Section~\ref{sec:cofactor}.

\begin{theorem}\label{thm:frames}
For every integer $n\geq 10$, there are positive integer matrices $B_n$ of size $(n-1)\times n$, such that $B_{10}=B$, every entry is at least two,
$B_n$ is the upper-left submatrix of $B_{n+1}$, and
\begin{equation}\label{eq:stablecofactors}
 \bgamma(B_n)=(\bc,\underbrace{0,\ldots,0}_{n-10}),\qquad
 \bgamma(\cube{B_n})=(\bd,\underbrace{0,\ldots,0}_{n-10}).
\end{equation}
Consequently, if $\by=(x_1,\ldots,x_{10},t_{11},\ldots,t_n)$, then
\begin{equation}\label{eq:stableequations}
 \det\begin{pmatrix}\by\\ B_n\end{pmatrix}
       =\sum_{j=1}^{10}c_jx_j,\qquad
 \det\cube{\begin{pmatrix}\by\\ B_n\end{pmatrix}}
       =\sum_{j=1}^{10}d_jx_j^3.
\end{equation}
which means the additional first-row entries $t_{11},\ldots,t_n$ are unrestricted by
both determinant equations.
\end{theorem}

The zeros in \eqref{eq:stablecofactors} are \emph{cofactors}, not matrix
entries. Every entry of $B_n$ remains at least two.

\begin{corollary}\label{cor:allcount}
Fix the lower rectangular block $B_n$ in Theorem~\ref{thm:frames}, and let
$\cB_n=\cB\times(2,3)^{n-10}$, with $\cB$ as in \eqref{eq:box}.
Then, for each fixed $n\geq10$,
\begin{align}\label{eq:allcount}
 &\#\left\{\by\in\Z^n:\by/P\in\cB_n,
  \ \det\begin{pmatrix}\by\\B_n\end{pmatrix}
   =\det\cube{\begin{pmatrix}\by\\B_n\end{pmatrix}}=1\right\}\notag\\
 &\hspace{30mm}=\kappa P^{n-4}+O_n(P^{n-4-1/100}),
\end{align}
where $\kappa>0$ is the same leading constant as in
Theorem~\ref{thm:seed}.
\end{corollary}

\begin{corollary}\label{cor:alllarge}
For every integer $r\geq1$ and every integer $n\geq10^r$, there is a
matrix of order $n$ with both determinants equal to one and every entry
at least $2^r$. Hence, for any $K\in \mathbb R$, all sufficiently
large orders admit such matrices with every entry greater than $K$.
\end{corollary}

The analytic part of the argument uses standard circle-method estimates; the main new ingredient is structural. The key step is the cofactor-preserving bordering construction of Theorem~\ref{thm:border}, which allows the order to be increased while retaining the required determinant structure. Its proof depends on two points: the relevant row-relation coefficients must remain nonzero, and the resulting inhomogeneous equations must be locally soluble at every prime. These issues are addressed in Sections ~\ref{sec:cone} and \ref{sec:border}.

\section{Cofactors and exact arithmetic certificates}\label{sec:cofactor}

For the block $B$ in \eqref{eq:B},
\begin{equation}\label{eq:cofdef}
 c_j=(-1)^{1+j}\det B_{\widehat j},\qquad
 d_j=(-1)^{1+j}\det(\cube{B})_{\widehat j}\quad(1\leq j\leq10).
\end{equation}
Laplace expansion along the first row gives
\begin{equation}\label{eq:LC}
 L(\bx):=\det A(\bx)=\sum_{j=1}^{10}c_jx_j,\qquad
 C(\bx):=\det\cube{A(\bx)}=\sum_{j=1}^{10}d_jx_j^3.
\end{equation}
The two cofactor vectors are as follows.

\begin{center}
\renewcommand{\arraystretch}{1.12}
\begin{tabular}{rrr}
\toprule
$j$ & $c_j$ & $d_j$\\
\midrule
1&4725&340739472931722111\\
2&18&28733076063275820\\
3&$-1432$&$-369435421691654920$\\
4&4682&124984744924291244\\
5&$-7509$&$-1235085791262920481$\\
6&$-7230$&$-722001405221789442$\\
7&504&791954706125920512\\
8&4797&183749257820706315\\
9&5531&484371424027128179\\
10&$-5416$&$-300611788966121134$\\
\bottomrule
\end{tabular}
\end{center}

Every $c_j$ and every $d_j$ is nonzero. The vectors are primitive:
\begin{equation}\label{eq:primitive}
 \gcd(c_1,\ldots,c_{10})=\gcd(d_1,\ldots,d_{10})=1.
\end{equation}
For the first identity one may use the particularly short certificate
\begin{equation}\label{eq:bezout}
 -159c_1+41658c_2-c_3=1;
\end{equation}
for the second, $\gcd(d_1,d_3)=1$ already suffices.

A stronger arithmetic property of $\bd$ is useful:
\begin{equation}\label{eq:pairgcd}
 \gcd(d_i,d_j)\mid840\quad(i\ne j),\qquad
 \lcm_{i<j}\gcd(d_i,d_j)=840.
\end{equation}
For completeness, all the pairwise gcds are displayed in the upper
triangular table below. Blank entries lie on or below the diagonal.
\begin{center}
\begin{tabular}{r*{9}{r}}
\toprule
$i\backslash j$&2&3&4&5&6&7&8&9&10\\
\midrule
1&3&1&1&3&3&3&3&1&1\\
2&&20&4&3&6&12&15&1&2\\
3&&&4&1&2&56&5&1&2\\
4&&&&1&2&4&1&1&2\\
5&&&&&3&3&3&1&1\\
6&&&&&&6&3&1&2\\
7&&&&&&&3&1&2\\
8&&&&&&&&1&1\\
9&&&&&&&&&1\\
\bottomrule
\end{tabular}
\end{center}
In particular, for every prime $p\geq11$, at least nine of the ten
coefficients $d_j$ are nonzero modulo $p$.

The accompanying script recomputes all twenty cofactors
by two different exact determinant algorithms, then checks
\eqref{eq:primitive}--\eqref{eq:pairgcd}. A description of the computation
appears in Appendix~\ref{app:verify}.

\subsection{A positive nonsingular real point}\hfill

Because $\br$ is a row of $B$, each of $A(\br)$ and
$\cube{A(\br)}$ has a repeated row. Hence
\begin{equation}\label{eq:realzero}
 L(\br)=C(\br)=0.
\end{equation}
For a common zero \(\bx\) of \(L\) and \(C\), write
\[J(\bx)=
\begin{pmatrix}
c_1 & \cdots & c_{10}\\
3d_1x_1^2 & \cdots & 3d_{10}x_{10}^2
\end{pmatrix}\] for the Jacobian matrix of \((L,C)\). We call \(\bx\) a nonsingular solution if \(\operatorname{rank}J(\bx)=2\), and singular solution otherwise.
For such $\bx$, introduce the reduced Jacobian minors
\begin{equation}\label{eq:minor}
 D_{ij}(\bx)=c_i d_jx_j^2-c_jd_ix_i^2.
\end{equation}
The Jacobian minor in columns $i,j$ is $3D_{ij}(\bx)$ and this $\bx$ is nonsingular if and only if at least one Jacobian minor is nonzero. Direct calculation gives
\begin{equation}\label{eq:realminor}
 D_{12}(\br)=1123741091386468277532\ne0.
\end{equation}
Thus $\br$ is a positive nonsingular real solution of $L=C=0$ lying in the box \eqref{eq:box}.

The homogeneous system appears after scaling. Setting $\bx=P\by$ transforms $L(\bx)=C(\bx)=1$ into $L(\by)=P^{-1}$ and $C(\by)=P^{-3}$. The limiting equations are therefore $L(\by)=C(\by)=0$.

\section{Solubility over every ring of \texorpdfstring{$p$}{p}-adic integers}\label{sec:local}
In the following sections, for a prime number \(p\), we write \(\mathbb F_p=\mathbb Z/p\mathbb Z\) for the finite field with \(p\) elements, \(\mathbb Z_p\) for the ring of \(p\)-adic integers, and \(\mathbb Q_p\) for its field of fractions.

We now verify that
\begin{equation}\label{eq:affine}
 L(\bx)=1,\qquad C(\bx)=1
\end{equation}
has a solution in $\Z_p^{10}$ at which its Jacobian has rank two over
$\Q_p$, for every prime $p$.

\subsection{A uniform finite-field criterion}

\begin{lemma}\label{lem:finitefield}
Let $p\geq13$ be prime. Suppose that $\bc\in\Fp^{10}$ is nonzero and
at least nine entries of $\bd\in\Fp^{10}$ are nonzero. Then
\[
 \sum c_jx_j=1,\qquad \sum d_jx_j^3=1
\]
has a solution in $\Fp^{10}$ at which the two gradients are linearly
independent.
\end{lemma}

\begin{proof}
Let $s$ be the number of nonzero $d_j$, so that $s\in\{9,10\}$.
Write $e_p(t)=\exp(2\pi it/p)$ and let $N$ count all solutions. 
By character orthogonality,
\[
1_{\{t=0\}}
=
\frac1p\sum_{a\in\mathbb F_p} e_p(at)
\qquad (t\in\mathbb F_p).
\]
Hence
\[
\begin{aligned}
N
&=
\sum_{x\in\mathbb F_p^{10}}
1_{\{L(x)=1\}}
1_{\{C(x)=1\}} =
p^{-2}
\sum_{a,b\in\mathbb F_p}
e_p(-a-b)
\sum_{x\in\mathbb F_p^{10}}
e_p\bigl(aL(x)+bC(x)\bigr).
\end{aligned}
\]
Since \(L\) and \(C\) are diagonal in the coordinates, the inner sum factors as
\[
\sum_{x\in\mathbb F_p^{10}}
e_p\bigl(aL(x)+bC(x)\bigr)
=
\prod_{j=1}^{10}
\sum_{x\in\mathbb F_p}
e_p(ac_jx+bd_jx^3).
\]
Therefore
\[
 N=p^{-2}\sum_{a,b\in\Fp}e_p(-a-b)
       \prod_{j=1}^{10}\sum_{x\in\Fp}e_p(ac_jx+bd_jx^3).
\]
The terms with \(b=0\) contribute exactly \(p^8\). Indeed, if
\(a\neq 0\), choose \(j\) with \(c_j\neq 0\); then
\[
\sum_{x\in\mathbb F_p} e_p(ac_jx)=0,
\]
so the corresponding product vanishes. When \(a=0\), all ten inner
sums equal \(p\), and hence this single term contributes
\(p^{-2}p^{10}=p^8\).

Now let \(b\neq0\). For every active index \(j\), so that \(d_j\neq0\),
the polynomial
\[
ac_jx+bd_jx^3
\]
has degree three with nonzero leading coefficient. The Weil bound \textcolor{red}{\cite{Weil}} gives
\[
\left|
\sum_{x\in\mathbb F_p}
e_p(ac_jx+bd_jx^3)
\right|
\le 2\sqrt p.
\]
For each of the remaining \(10-s\) indices, we use the trivial bound
\(p\). Since there are fewer than \(p^2\) pairs \((a,b)\) with
\(b\neq0\), it follows that
\[
|N-p^8|
\le
p^{-2}\,p^2(2\sqrt p)^s p^{10-s}
=
2^s p^{10-s/2}.
\]
Thus, since \(s\in\{9,10\}\),
\[
|N-p^8|
\le 512p^{11/2}.
\]
For \(s=10\), the last inequality uses \(2/\sqrt p\le1\).

At a singular solution, since $\bc\ne0$, there is a $\lambda\in\Fp$
with
\[
 3d_jx_j^2=\lambda c_j\quad(1\leq j\leq10).
\]
Call an index active if $d_j\ne0$ in $\Fp$, and inactive otherwise. There
are $s$ active indices. For fixed $\lambda$, an active coordinate satisfies
\[
 x_j^2=\frac{\lambda c_j}{3d_j}
\]
and therefore has at most two choices for $x_j$. We bound each inactive coordinate
trivially by $p$ choices. Summing this upper bound over the $p$ possible
values of $\lambda$ shows that the number of singular solutions is at most
\[
 p\cdot2^sp^{10-s}=2^s p^{11-s}\leq512p^2.
\]
The final inequality uses
$2^{10}p\leq 2^9p^2$, which holds for $p\geq13$. The number of nonsingular solutions is therefore at least
\begin{equation}\label{eq:finitepositive}
 p^8-512p^{11/2}-512p^2,
\end{equation}
which is positive for every $p\geq13$.
\end{proof}


\subsection{The exceptional primes}\hfill

For $p=2,3,5,7,11$, the following table supplies explicit certificates.
For each row, both $L(\bx)$ and $C(\bx)$ are congruent to $1$ modulo
$m$. The last column of the table is $D_{ij}(\bx)\bmod p$, not the full Jacobian
minor.

\begin{center}
\small
\renewcommand{\arraystretch}{1.25}
\begin{tabular}{rrlcc}
\toprule
$p$&$m$&$\bx\pmod m$&$(i,j)$&$D_{ij}\bmod p$\\
\midrule
2&2&$(0,0,1,0,1,1,0,1,1,0)$&$(1,5)$&1\\
3&27&$(8,16,18,24,1,6,9,23,19,16)$&$(3,9)$&1\\
5&5&$(3,0,2,2,2,3,3,1,2,2)$&$(1,2)$&3\\
7&7&$(0,4,0,1,2,1,4,2,1,2)$&$(2,3)$&5\\
11&11&$(7,2,10,5,2,7,1,7,4,8)$&$(1,3)$&2\\
\bottomrule
\end{tabular}
\end{center}

For $p=2,5,7,11$, the full minor $3D_{ij}$ is a unit modulo $p$, so ordinary Hensel lifting applies. At $p=3$, every entry of the cubic gradient is divisible by three. We instead start from the solution modulo \(27\) given in the table and apply the following lifting argument.

\begin{lemma}\label{lem:three}
The $p=3$ certificate in the table lifts to a solution of
\eqref{eq:affine} in $\Z_3^{10}$ at which the Jacobian has rank two over
$\Q_3$.
\end{lemma}

\begin{proof}
Let $\bx^0$ be the displayed certificate vector in the table. Fix its eight coordinates other than $x_3,x_9$, and put $t=x_9$. Since $c_3=-1432$ is a $3$-adic unit, the linear equation can be solved exactly by setting
\[
 z_3(t)=x_3=\frac{1-\sum_{k\ne3,9}c_kx_k^0-c_9t}{c_3}\in\Z_3[t].
\]
Here $z_3'(t)=-c_9/c_3$. Since the displayed certificate satisfies the
linear congruence modulo $27$, substituting $t_0=x_9^0$ gives
$z_3(t_0)\equiv x_3^0\pmod {27}$.
The cubic equation $C(\bx) = 1$ becomes $F(t)=0$, where
\[
 F(t)=d_3z_3(t)^3+d_9t^3+\sum_{k\ne3,9}d_k(x_k^0)^3-1.
\]
At $t_0=19$, the cubic congruence modulo $27$ implies
\[
 F(t_0)\equiv0\pmod{27}.
\]
Furthermore,
\[
 F'(t)=\frac{3}{c_3}
       \bigl(c_3d_9t^2-c_9d_3z_3(t)^2\bigr).
\]
With \(x_3=z_3(t)\) and \(x_9=t\), the expression in parentheses is \(D_{39}(\bx)\). At \(t=t_0\), the certificate gives \(D_{39}(\bx)\equiv1\pmod3\). Therefore
\[
 v_3(F(t_0))\geq3>2v_3(F'(t_0))=2.
\]
The congruences and valuation assertions used here are verified by the
exact-arithmetic computation described in Appendix~\ref{app:verify}. The strong form of Hensel's lemma gives a root $t\in\Z_3$ with $t\equiv t_0\pmod9$. Set $x_9=t$, $x_3=z_3(t)$, and retain the other eight coordinates of $\bx^0$. The definition of $z_3$ gives the linear equation exactly, while $F(t)=0$ gives the cubic equation. The reduced minor remains a unit modulo $3$ at this point; hence the corresponding full Jacobian minor has valuation exactly one and is nonzero in $\Q_3$.
\end{proof}

\begin{proposition}\label{prop:alllocal}
The system \eqref{eq:affine} has a nonsingular solution in $\Z_p^{10}$ for every prime $p$.
\end{proposition}

\begin{proof}
For every prime \(p\ge 13\), the primitivity of \(c\) in \eqref{eq:primitive} gives \(c\not\equiv 0 \pmod p\), while \eqref{eq:pairgcd} implies that at most one of the \(d_j\) vanishes modulo \(p\). Thus the hypotheses of Lemma~\ref{lem:finitefield} are satisfied, so there is a nonsingular solution modulo \(p\). Ordinary Hensel lifting then gives a solution in \(\mathbb Z_p^{10}\) at which the Jacobian has rank two over \(\mathbb Q_p\).

For \(p=2,5,7,11\), the explicit certificates in the table have a Jacobian minor \(3D_{ij}\) that is a unit modulo \(p\), and hence lift by the ordinary multivariable Hensel lemma. The remaining prime \(p=3\) is handled by Lemma~\ref{lem:three}.
\end{proof}

\section{An inhomogeneous cubic--linear counting theorem}\label{sec:analytic}

Now we state the analytic result in a slightly more general form. All its coefficients and boxes are fixed as $P$ tends to infinity.

\begin{theorem}\label{thm:analytic}
Let $s\geq10$, $c_j,d_j\in\Z\setminus\{0\}$ for $1\leq j\leq s$, and $u,v\in\Z$. Let
\[
 L(\bx)=\sum_{j=1}^s c_jx_j,\qquad C(\bx)=\sum_{j=1}^s d_jx_j^3.
\]
Let $\cB=\prod_{j=1}^s I_j$, where $I_j=(\ell_j,h_j)$ and $0<\ell_j<h_j<\infty$. Then we define
\begin{equation}\label{eq:generalthm}
 R(P):=\#\{\bx\in\Z^s:\bx/P\in\cB,\ L(\bx)=u,\ C(\bx)=v\}
 =\fS(u,v)\fJ(\cB)P^{s-4}+O(P^{s-4-1/100}).
\end{equation}
The singular series $\fS(u,v)$ and singular integral $\fJ(\cB)$ defined in \eqref{eq:singseries} and \eqref{eq:singint} are absolutely convergent. Their product is positive if the following two conditions hold:
\begin{enumerate}
\item $L(\bx)=u$, $C(\bx)=v$ has a solution in $\Z_p^s$ with Jacobian rank two over $\Q_p$, for every prime $p$;
\item $L(\boldsymbol{y})=C(\boldsymbol{y})=0$ has a real solution in $\cB$ with Jacobian rank two.
\end{enumerate}
The implicit constant may depend on $s$, the coefficients, $u,v$ and the box $\cB$. 
\end{theorem}

We use the following three standard exponential-sum estimates. With $e(t)=\exp(2\pi it)$, Wooley's theorem \textcolor{red}{\cite[Theorem 1.1]{Wooley}} states that, for every $\varepsilon>0$,
\begin{equation}\label{eq:momentinput}
 \int_0^1\int_0^1
 \left|\sum_{1\leq x\leq X}e(\alpha x+\beta x^3)\right|^9
 \,d\alpha\,d\beta\ll_\varepsilon X^{5+\varepsilon}.
\end{equation}
We also use the classical cubic Weyl inequality \textcolor{red}{\cite[Lemma 2.4]{Vaughan}}: when $(a,q)=1$ and $|\theta-a/q|\leq \dfrac{1}{q^2}$,
\begin{equation}\label{eq:weylinput}
 \sum_{x\in PI\cap\Z}e(\theta x^3+\lambda x)
 \ll_{I,\varepsilon}P^{1+\varepsilon}
       (\dfrac{1}{q}+\dfrac{1}{P}+\dfrac{q}{P^3})^{1/4},
\end{equation}
uniformly in the lower-degree coefficients, and hence in the linear coefficient $\lambda$. Finally, the complete-polynomial-sum estimate \textcolor{red}{\cite[Theorem 7.1]{Vaughan}}, specialized to degree three, gives
\begin{equation}\label{eq:completeinput}
 \sum_{r\bmod q}e\bigl((ar+br^3)/q\bigr)
 \ll_\varepsilon q^{2/3+\varepsilon}
 \quad\text{if }\gcd(a,b,q)=1.
\end{equation}

\subsection{Generating functions and the ninth moment}\hfill

Write $\T=\R/\Z$ and set
\begin{equation}\label{eq:fj}
 f_j(\alpha,\beta)=\sum_{x\in PI_j\cap\Z}
                      e(\alpha c_jx+\beta d_jx^3).
\end{equation}

Using the orthogonality relation
\[
\int_{\mathbb T} e(\alpha n)\,d\alpha
=
\mathbf 1_{\{n=0\}}
\qquad (n\in\mathbb Z),
\]
for the two equations \(L(x)=u\) and \(C(x)=v\), we obtain

\begin{equation}\label{eq:orthog}
R(P)=\sum_{\substack{x_j\in PI_j\cap\mathbb Z\\1\le j\le s}}
\int_{\mathbb T^2}
e(\alpha(L(x)-u)+\beta(C(x)-v)\bigr)
\,d\alpha\,d\beta =\int_{\mathbb T^2}e(-u\alpha-v\beta)\prod_{j=1}^s f_j(\alpha,\beta)\, d\alpha\,d\beta.
\end{equation}

Because $c_jd_j\ne0$, the map
$(\alpha,\beta)\mapsto(c_j\alpha,d_j\beta)$ preserves normalized Haar
measure on $\T^2$, which satisfies $\operatorname{meas}(\T^2)=1$. An interval sum in \eqref{eq:fj} is the difference of
two initial sums of lengths $O(P)$. Minkowski's inequality and
\eqref{eq:momentinput} give
\begin{equation}\label{eq:fjmoment}
 \int_{\T^2}|f_j(\alpha,\beta)|^9\,d\alpha\,d\beta
 \ll P^{5+\varepsilon}.
\end{equation}
In particular, H\"older's inequality yields
\begin{equation}\label{eq:productnine}
 \int_{\T^2}\prod_{j=2}^{10}|f_j(\alpha,\beta)|\,d\alpha\,d\beta
 \ll P^{5+\varepsilon}.
\end{equation}
For $s>10$, the additional $s-10$ factors may be trivially bounded by
$O(P)$.

\subsection{A two-dimensional major-arc dissection}\hfill

Set
\begin{equation}\label{eq:QH}
 Q=P^{1/20},\qquad D=|d_1|,\qquad H=|c_1d_1|.
\end{equation}
All distances in this subsection are distances modulo one. For integers
$q,a,b$ with $1\leq q\leq HQ$, $0\leq a,b<q$ and $gcd(a,b,q)=1$, let
\begin{equation}\label{eq:majordef}
 \cM(q,a,b)=\{(\alpha,\beta)\in\T^2:
  |\alpha-\dfrac{a}{q}|\leq \frac{Q^2}{P},\quad
  |\beta-\dfrac{b}{q}|\leq \frac{Q}{P^3}\}.
\end{equation}
Let $\mathfrak M$ denote the union of these rectangles, called the \emph{major arcs}, and let $\mathfrak m=\mathbb T^2\setminus\mathfrak M$ denote their complement, called the \emph{minor arcs}.
For sufficiently large $P$, these rectangles are pairwise
disjoint. Indeed, distinct reduced rational pairs with
denominators $q,q'\le HQ$ differ in at least one coordinate,
modulo one, by at least
\[
\frac{1}{qq'}\ge \frac{1}{(HQ)^2}.
\]
Since $Q=P^{1/20}$ and $H$ is fixed, both half-widths
$Q^2/P$ and $Q/P^3$ are $o(Q^{-2})$, so the rectangles
are disjoint for sufficiently large $P$.

Moreover, the number of rectangles is at most
\[
\sum_{q\le HQ}q^2\ll Q^3.
\]
Each rectangle has measure at most $4Q^3/P^4$.
Consequently,
\begin{equation}\label{eq:majorvolume}
\operatorname{meas}(\mathfrak M)
\ll Q^3\cdot\frac{Q^3}{P^4}
=\frac{Q^6}{P^4}.
\end{equation}

\begin{lemma}\label{lem:minorpointwise}
Uniformly for $(\alpha,\beta)\in\cm$,
\begin{equation}\label{eq:minorpointwise}
 |f_1(\alpha,\beta)|\ll P^{1+\varepsilon}Q^{-1/4}.
\end{equation}
\end{lemma}

\begin{proof}
Let $\mathfrak U$ be the union of the intervals in the $\beta$ coordinate
\[
 |\beta-\dfrac{b}{q}|\leq \dfrac{Q}{P^3},\qquad
 1\leq q\leq DQ,\quad(b,q)=1.
\]
First, suppose that $\beta\notin\mathfrak U$. By Dirichlet approximation theorem, there exist coprime integers $b',q'$ such that
\[
 1\leq q'\leq \dfrac{P^3}{Q},\qquad
 |d_1\beta-\dfrac{b'}{q'}|\leq\dfrac{Q}{q'P^3}.
\]
If $q'\leq Q$, dividing by $d_1$ and reducing the resulting fraction yields coprime integers $b_0,q_0$ satisfying
\[
 q_0\leq |d_1|q'\leq DQ,
 \qquad
 |\beta-\frac{b_0}{q_0}|
 \leq \frac{Q}{|d_1|q'P^3}
 \leq \frac{Q}{P^3}.
\]
This contradicts $\beta\notin\mathfrak U$. Hence $q'>Q$. The upper bound $q'\leq P^3/Q$ also implies that the approximation error is at most $\frac{1}{(q')^{2}}$ and 
\[
 \dfrac{1}{q'}+\dfrac{1}{P}+\dfrac{q'}{P^3}\ll \dfrac{1}{Q}.
\]
Applying the Weyl estimate \eqref{eq:weylinput}, we obtain
\eqref{eq:minorpointwise}.

Now suppose that $\beta\in\mathfrak U$, and write
\[
\beta=\dfrac{b}{q}+\xi,
\qquad q\le DQ,\qquad |\xi|\le \frac{Q}{P^3}.
\]
If $\alpha$ were within $\frac{Q^2}{P}$ of a
rational number with denominator $|qc_1|$, then the rational number and $b/q$ would have a common denominator at most
$|qc_1|\leq HQ$. Thus $(\alpha,\beta) \in \cM$.
Consequently, outside $\cM$,
\begin{equation}\label{eq:lineardistance}
 \|qc_1\alpha\|\geq |qc_1|\frac{Q^2}{P}.
\end{equation}
Split $f_1$ into residue classes modulo $q$.
Since $x^3\equiv r^3\pmod q$ whenever
$x\equiv r\pmod q$, we have
\[
f_1(\alpha,\beta)
=
\sum_{r\bmod q}
e\left(\frac{bd_1r^3}{q}\right)
\sum_{\substack{x\in PI_1\cap\mathbb Z\\
                 x\equiv r\pmod q}}
e(c_1\alpha x+d_1\xi x^3).
\]

For each fixed residue class, writing $x=r+qk$,
the linear exponential factor forms a geometric
progression. Consequently, for any subinterval
$J\subseteq PI_1$, we have
\[
\left|
\sum_{\substack{x\in J\cap\mathbb Z\\
                 x\equiv r\pmod q}}
e(c_1\alpha x)
\right|
\ll
\min\left\{
\frac Pq+1,
\frac{1}{\|qc_1\alpha\|}
\right\}.
\]

Moreover, the total variation of
$e(d_1\xi x^3)$ on $PI_1$ is
$O(P^3|\xi|)$.
Partial summation therefore gives
\[
|f_1(\alpha,\beta)|
\ll
q(1+P^3|\xi|)
\min\left\{
\frac Pq+1,
\frac{1}{\|qc_1\alpha\|}
\right\}.
\]
The factor $1+P^3|\xi|$ bounds the total variation of
$e(\xi d_1x^3)$ on $PI_1$. Since $q\ll Q\ll P$, equation
\eqref{eq:lineardistance} gives
\[
 |f_1(\alpha,\beta)|
 \ll q(1+Q)\frac{P}{|qc_1|Q^2}
 \ll \frac{P}{Q},
\]
This implies \eqref{eq:minorpointwise} and completes the proof.
\end{proof}

Combining \eqref{eq:productnine} and \eqref{eq:minorpointwise} with
trivial estimates for the remaining factors, we have
\begin{equation}\label{eq:minorintegral}
 \int_{\cm}\left|\prod_{j=1}^s f_j(\alpha,\beta)\right|
       \,d\alpha\,d\beta
 \ll P^{s-4+\varepsilon}Q^{-1/4}.
\end{equation}
The fixed-target phase in \eqref{eq:orthog} has absolute value one and
therefore does not affect this estimate.

\subsection{Complete sums and major-arc approximation}\hfill

Define
\begin{align}
 S_j(q,a,b)&=\sum_{r\bmod q}
          e\bigl(\dfrac{ac_jr+bd_jr^3}{q}\bigr),\label{eq:Sj}\\
 V_j(z,w)&=\int_{I_j}e(c_jzy+d_jwy^3)\,dy.\label{eq:Vj}
\end{align}
Suppose that $gcd(a,b,q)=1$, and let $g_j=\gcd(q,ac_j,bd_j)$.
For each prime $p\mid q$, at least one of $a,b$ is coprime to $p$. Thus, $v_p(g_j)\le \max{(v_p(c_j),v_p(d_j))}$,
and therefore $g_j\mid |c_jd_j|$. 
Let
\[
 q_j=\dfrac{q}{g_j},\qquad a_j'=\dfrac{ac_j}{g_j},\qquad b_j'=\dfrac{bd_j}{g_j}.
\]
Then $\gcd(q_j,a_j',b_j')=1$. As the summand in \eqref{eq:Sj} is periodic modulo $q_j$,
\[
 S_j(q,a,b)
 =g_j\sum_{r\bmod q_j}
 e\bigl(\dfrac{a_j'r+b_j'r^3}{q_j}\bigr).
\]
Applying \eqref{eq:completeinput} and $g_j\mid |c_jd_j|$, we obtain
\begin{equation}\label{eq:Sbound}
 S_j(q,a,b)\ll_{\varepsilon}g_jq_j^{2/3+\varepsilon}\ll_{c_j,d_j,\varepsilon}q^{2/3+\varepsilon}.
\end{equation}
The bound is uniform in $q,a,b$, without the assumption of coprimality between $q$ and $c_jd_j$.

On a major arc, write 
\[\alpha=\dfrac{a}{q}+\eta,\qquad \beta=\dfrac{b}{q}+\xi.
\]
Decomposing $f_j$ into residue classes modulo $q$ and approximating each sum by  its integral gives
\begin{equation}\label{eq:majorapprox}
 f_j(\alpha,\beta)
 =\dfrac{1}{q}S_j(q,a,b)P V_j(P\eta,P^3\xi)
   +O\bigl(q(1+P|\eta|+P^3|\xi|)\bigr).
\end{equation}
This follows from the elementary estimate
\[
 \sum_{\substack{x\in PI_j\cap\Z\\x\equiv r\ (q)}}w(x)
 =q^{-1}\int_{PI_j}w(t)\,dt
       +O\left(1+\int_{PI_j}|w'(t)|\,dt\right),
\]
applied to $w(t)=e(c_j\eta t+d_j\xi t^3)$ and summed over residues.
Under the bounds in \eqref{eq:majordef}, the error in \eqref{eq:majorapprox} is $O(Q^3)$. Denote the main term in \eqref{eq:majorapprox} by $f_j^*$. Both $f_j$ and $f_j^*$ are $O(P)$. Then the telescoping identity for the difference of two products gives, uniformly on $\cM$,
\[
 \left|\prod_{j=1}^s f_j-\prod_{j=1}^s f_j^*\right|
 \ll \sum_{j=1}^s |f_j-f_j^*|P^{s-1}
 \ll P^{s-1}Q^3.
\]
By \eqref{eq:majorvolume}, the error after integrating over $\cM$ is 
\begin{equation}\label{eq:approxerror}
 O(P^{s-1}Q^3\operatorname{meas}(\cM))
 =O(P^{s-5}Q^9).
\end{equation}
Define
\begin{equation}\label{eq:Aq}
 \mathcal A(q;u,v)=q^{-s}
 \sum_{\substack{a,b\bmod q\\(a,b,q)=1}}
 e\bigl(-\dfrac{ua+vb}{q}\bigr)\prod_{j=1}^s S_j(q,a,b),
\end{equation}
with $\mathcal A(1;u,v)=1$, and let
\begin{equation}\label{eq:Jtrunc}
 J_Q(P)=\int_{-Q^2}^{Q^2}\int_{-Q}^{Q}
 e(-\dfrac{uz}{P}-\dfrac{vw}{P^3})\prod_{j=1}^s V_j(z,w)\,dw\,dz.
\end{equation}
On each major arc $\mathfrak M(q,a,b)$, 
\[
e(-u\alpha-v\beta)
=
e\left(-\frac{ua+vb}{q}\right)e(-u\eta-v\xi),
\]
and
\[
\prod_{j=1}^s f_j^*(\alpha,\beta)
=
\frac{P^s}{q^s}
\prod_{j=1}^s S_j(q,a,b)
\prod_{j=1}^s V_j(P\eta,P^3\xi).
\]
Making the change of variables
$z=P\eta$ and $w=P^3\xi$, so that
$d\alpha\,d\beta=P^{-4}\,dz\,dw$, gives
\[
\begin{aligned}
&\int_{\mathfrak M(q,a,b)}
e(-u\alpha-v\beta)
\prod_{j=1}^s f_j^*(\alpha,\beta)
\,d\alpha\,d\beta\\
&\quad=
\frac{P^{s-4}}{q^s}
e\left(-\frac{ua+vb}{q}\right)
\prod_{j=1}^s S_j(q,a,b)\,J_Q(P).
\end{aligned}
\]
Summing over all major arcs and applying
the integrated error estimate (5.19),
we obtain
\begin{align}\label{eq:majorintegral}
 &\int_{\cM}e(-u\alpha-v\beta)\prod_{j=1}^s f_j(\alpha,\beta)
          \,d\alpha\,d\beta\notag\\
 &\quad=P^{s-4}\left(\sum_{q\leq HQ}\mathcal A(q;u,v)\right)J_Q(P)
           +O(P^{s-5}Q^9).
\end{align}

\subsection{The singular integral}\hfill

Since $I_j=(\ell_j,h_j)$ with $0<\ell_j<h_j<\infty$, the phase
\[
\phi_j(y)=c_jzy+d_jwy^3
\]
satisfies
\[
\phi_j'(y)=c_jz+3d_jwy^2, \qquad \phi_j''(y)=6d_jwy.
\]
If $|z|\ge K|w|$, where $K$ is a sufficiently large fixed constant, then $|\phi_j'(y)|\gg |z|$ throughout $I_j$. The first-derivative estimate gives
\[
|V_j(z,w)|\ll \frac{1}{1+|z|}.
\]
If $|z|<K|w|$, then $|\phi_j''(y)|\gg |w|$ throughout $I_j$, and the second-derivative estimate gives
\[
|V_j(z,w)|\ll \frac{1}{(1+|w|)^{1/2}}.
\]
Combining these estimates with the trivial bound near the origin, we obtain
\begin{equation}\label{eq:Vdecay}
 |V_j(z,w)|\ll \dfrac{1}{(1+|z|+|w|)^{1/2}}.
\end{equation}
Define
\begin{equation}\label{eq:singint}
 \fJ(\cB)=\int_{\R^2}\prod_{j=1}^s V_j(z,w)\,dw\,dz
\end{equation}
For $s\geq10$, the integrand satisfies
$$\left|\prod_{j=1}^{s}V_j(z,w)\right|
\ll \frac{1}{(1+|z|+|w|)^5}.$$
Thus, \eqref{eq:singint} converges absolutely and
$$\int_{\mathbb R^2}
(|z|+|w|)
\left|\prod_{j=1}^{s}V_j(z,w)\right|
\,dw\,dz<\infty.$$
Outside the rectangle $[-Q^2,Q^2]\times[-Q,Q]$, we have $|z|+|w|\ge Q$, so the contribution from this region is $O(\frac{1}{Q^3})$. Moreover,
$$
\left|
e\left(-\frac{uz}{P}-\frac{vw}{P^3}\right)-1
\right|
\ll_{u,v}\frac{|z|+|w|}{P}.
$$
Then the preceding moment bound yields
\begin{equation}\label{eq:Japprox}
 J_Q(P)=\fJ(\cB)+O(\dfrac{1}{Q^3}+\dfrac{1}{P}).
\end{equation}

To prove positivity, define
\[
 \Phi(\boldsymbol{y})=(L(\boldsymbol{y}),C(\boldsymbol{y})),
\]
and let $\mu$ be the pushforward of Lebesgue measure on
$\cB$ under $\Phi$. With convention, 
\[
\widehat\mu(z,w)=\int_{\mathbb R^2}e(zt_1+wt_2)\,d\mu(t_1,t_2),
\]
Fubini’s theorem implies
\[
 \widehat\mu(z,w) =\int_{\cB}e\bigl(zL(\boldsymbol{y})+wC(\boldsymbol{y})\bigr)
       \,d\boldsymbol{y}
 =\prod_{j=1}^sV_j(z,w).
\]
By \eqref{eq:Vdecay}, \(\widehat\mu\in L^1(\mathbb R^2)\), then the Fourier inversion shows that $\mu$ has a
continuous density $\rho$, with
$$\rho(0,0)=\fJ(\cB).$$
By the real hypothesis in Theorem~\ref{thm:analytic}, there is a point $\boldsymbol{y}_0$ in the interior of $\cB$ such that
\[
\Phi(\boldsymbol{y}_0)=(0,0), \qquad \text{rank } D\,\Phi(\boldsymbol{y}_0)=2.
\]
After relabeling coordinates, we may assume that the Jacobian minor corresponding to $y_1,y_2$ is nonzero. Then the map
\[
 \Psi(\boldsymbol{y})=
   \bigl(\Phi(\boldsymbol{y}),y_3,\ldots,y_s\bigr)
\]
is a local diffeomorphism at $\boldsymbol{y}_0$. 

Choose a neighborhood $U$ of $\boldsymbol{y_0}$, contained in $\cB$, on which $\Psi$ is a diffeomorphism. Its image contains a product $N\times W$, where $N$ is a neighborhood of $(0,0)$ and $W$ is a positive volume box in the remaining coordinates. Restricting to $\Psi^{-1}(N\times W)$, the change-of-variable formula gives the density 
\[
\rho_U(t_1,t_2)=\int_{W}|\det D \Psi^{-1}(t_1,t_2, \mathbf t')|\,d\mathbf t'
\]
where \(\mathbf t'=(t_3,\ldots,t_s)\) represents the remaining coordinates. The density is continuous and strictly positive on $N$. As $\rho\geq\rho_U$ on $N$, it follows that 
\begin{equation}\label{eq:Jpositive}
 \fJ(\cB)=\rho(0,0)>0.
\end{equation}

\subsection{The singular series and its local factors}\hfill

Applying \eqref{eq:Sbound} with $\dfrac{\varepsilon}{s}$ in place of $\varepsilon$ and summing over at most $q^2$ pairs $(a,b)$, we have 
\begin{equation}\label{eq:Aqbound}
 |\mathcal A(q;u,v)|\ll q^{2-s/3+\varepsilon}
                    \leq q^{-4/3+\varepsilon}.
\end{equation} Choosing $0<\varepsilon<1/3$, we have
\[
\sum_{q=1}^{\infty}|\mathcal A(q;u,v)|
\ll
\sum_{q=1}^{\infty}q^{-4/3+\varepsilon}
<\infty.
\]
Thus the singular series
\begin{equation}\label{eq:singseries}
\mathfrak S(u,v)
=
\sum_{q=1}^{\infty}\mathcal A(q;u,v)
\end{equation}
converges absolutely. Moreover, by the integral test,
\[
\sum_{q>HQ}|\mathcal A(q;u,v)|
\ll
\sum_{q>HQ}q^{-4/3+\varepsilon}
\ll Q^{-1/3+\varepsilon},
\]
since $H$ is fixed. Consequently,
\begin{equation}\label{eq:Sapprox}
\sum_{q\le HQ}\mathcal A(q;u,v)
=
\mathfrak S(u,v)+O(Q^{-1/3+\varepsilon}).
\end{equation}

Let $N(q;u,v)$ be the number of solutions of $L(\bx)\equiv u$,
$C(\bx)\equiv v\pmod q$. By orthogonality of additive characters,
\[
\begin{aligned}
q^{-(s-2)}N(q;u,v)
&=q^{-s}\sum_{a,b\bmod q}
e\left(-\frac{ua+vb}{q}\right)
\prod_{j=1}^s S_j(q,a,b).
\end{aligned}
\]
For each pair $(a,b)$, write
$g=\gcd(a,b,q)$ and $d=q/g$.
Then $a=ga'$, $b=gb'$,
$\gcd(a',b',d)=1$, and
\[
S_j(q,a,b)=gS_j(d,a',b').
\]
Grouping the frequencies according to $d\mid q$
therefore gives
\begin{equation}\label{eq:divisoridentity}
q^{-(s-2)}N(q;u,v)
=\sum_{d\mid q}\mathcal A(d;u,v).
\end{equation}

If $\gcd(m,n)=1$, the Chinese remainder theorem
gives a bijection between solutions modulo $mn$
and pairs of solutions modulo $m$ and $n$. Hence
\[
N(mn;u,v)=N(m;u,v)N(n;u,v),
\]
so the left-hand side of
\eqref{eq:divisoridentity} is multiplicative in $q$.
By Mobius inversion,
\[
\mathcal A(q;u,v)
=
\sum_{d\mid q}
\mu(q/d)\,d^{-(s-2)}N(d;u,v),
\]
which is also multiplicative.
Consequently, absolute convergence yields,
\begin{equation}\label{eq:Euler}
 \fS(u,v)=\prod_p\sigma_p(u,v),\qquad
\sigma_p(u,v)=\sum_{k=0}^{\infty}\mathcal A(p^k;u,v)
 =\lim_{k\to\infty}p^{-(s-2)k}N(p^k;u,v),
\end{equation}
by taking $q=p^k$ in \eqref{eq:divisoridentity}.

To establish positivity of the local factors, we use the following form of Hensel's lifting Lemma, which only requires a
$p$-adic solution with Jacobian rank two over $\mathbb Q_p$.
This weaker condition is essential at $p=3$, where the cubic
gradient vanishes modulo $3$, so no solution can be
nonsingular modulo $3$.

\begin{lemma}\label{lem:densitypositive}
Let \(F_1,F_2\in\mathbb Z[x_1,\ldots,x_s]\), and suppose that
\[
F_1(\mathbf x^{(0)})=F_2(\mathbf x^{(0)})=0
\]
for some \(\mathbf x^{(0)}\in\mathbb Z_p^s\). Suppose that a two-column minor of the Jacobian at \(\mathbf x^{(0)}\) has finite valuation \(h\). Then, for every \(k\ge2h+1\), the number of solutions modulo \(p^k\) is at least
\[
p^{(s-2)(k-2h-1)}.
\]
In particular, its normalized local density, if it exists, is positive.
\end{lemma}

\begin{proof}
Relabel the variables so that the Jacobian minor formed by the first two columns has
valuation $h$. Choose arbitrary $y_3,\ldots,y_s\in\Z_p$ satisfying
\[
 y_i\equiv x_i^{(0)}\pmod {p^{2h+1}},
 \qquad3\leq i\leq s.
\]
With these coordinates fixed, write $\boldsymbol{x}'=(x_1, x_2)$ and define
\[
G(\boldsymbol{x}')=
\begin{pmatrix}
F_1(x_1,x_2,y_3,\ldots,y_s)\\
F_2(x_1,x_2,y_3,\ldots,y_s)
\end{pmatrix},
\]
and let \(J(\boldsymbol{x}')\) be its \(2\times2\) Jacobian matrix.
At the initial point $\boldsymbol{x}'^{(0)}=(x_1^{(0)},x_2^{(0)})$,
we have $G(\boldsymbol{x}'^{(0)})\in p^{2h+1}\mathbb Z_p^2$. Moreover, the determinant of \(J(\boldsymbol{x}'^{(0)})\) is congruent modulo \(p^{2h+1}\) to the original minor, so its valuation remains \(h\).

Suppose that, at some stage, $G(\boldsymbol{x}')\in p^m\Z_p^2$ with $m>2h$ and $v_p(\det J(\boldsymbol{x}'))=h$. 
The Newton correction is
$\boldsymbol{\delta}=-J(\boldsymbol{x}')^{-1}G(\boldsymbol{x}'),$
Since $J^{-1}=\operatorname{adj}(J)/\det J$ and the adjugate has entries in \(\mathbb Z_p\), $\boldsymbol\delta\in p^{m-h}\mathbb Z_p^2$.
Taylor expansion gives
\[
G(\boldsymbol{x}'+\boldsymbol{\delta})\in p^{2(m-h)}\mathbb Z_p^2.
\]
Because \(m-h\ge h+1\), every entry of
$J(\mathbf x'+\boldsymbol\delta)-J(\mathbf x')$
is divisible by \(p^{h+1}\). Consequently,
$\det J(\mathbf x'+\boldsymbol\delta)
\equiv\det J(\mathbf x')\pmod{p^{h+1}}$, so the determinant retains valuation \(h\).

Starting with \(m_0=2h+1\), we may repeat the Newton step with $m_{n+1}=2(m_n-h)$. Since \(m_n=2h+2^n\), the residuals tend to zero and the
corrections tend to zero in the \(p\)-adic metric. The iterations therefore converge in \(\mathbb Z_p^2\) to a zero of \(G\).

For $k\geq2h+1$, each free coordinate has
$p^{k-(2h+1)}$ possible residues modulo $p^k$ in its prescribed
congruence class. Thus, there are
$p^{(s-2)(k-2h-1)}$ choices for the free residues. Choose a \(p\)-adic
representative of each choice and apply the lifting
argument above. The resulting solutions are distinct
modulo \(p^k\), since their free coordinates are distinct.
This proves the lower bound.

If \(N_F(p^k)\) denotes the number of solutions modulo
\(p^k\), then
$$
\frac{N_F(p^k)}{p^{(s-2)k}}
\ge\frac{1}{p^{(s-2)(2h+1)}}.
$$
Thus the normalized local density is positive whenever the limit exists.
\end{proof}

Applying Lemma \ref{lem:densitypositive} to \(F_1=L-u\) and \(F_2=C-v\), under condition $(1)$ of Theorem~\ref{thm:analytic}, we obtain $\sigma_p(u,v)>0$ for every $p$. Moreover, absolute convergence of \eqref{eq:singseries} gives
\[
\sum_p|\sigma_p(u,v)-1|
\le
\sum_p\sum_{k=1}^{\infty}
|\mathcal A(p^k;u,v)|
\le
\sum_{q=2}^{\infty}|\mathcal A(q;u,v)|
<\infty.
\]
Hence, the Euler product in
\eqref{eq:Euler} is strictly positive, so $\fS(u,v)>0$.

\subsection{Completion of the analytic proof}\hfill

Combining the minor-arc estimate \eqref{eq:minorintegral} with the major-arc formula \eqref{eq:majorintegral}, and using 
\eqref{eq:Japprox} and \eqref{eq:Sapprox}, we obtain
\begin{align}\label{eq:allerrors}
 R(P)&={}\fS(u,v)\fJ(\cB)P^{s-4}\notag\\
 &+O\bigl(P^{s-4+\varepsilon}Q^{-1/4}
       +P^{s-5}Q^9
        +P^{s-4}Q^{-1/3+\varepsilon}\notag
        +P^{s-4}Q^{-3}+P^{s-5}\bigr).
\end{align}
With $Q=P^{1/20}$ and, for example, a sufficiently small $\varepsilon<1/2000$, each error term is $O(P^{s-4-1/100})$. Under conditions (1) and (2) of Theorem~\eqref{eq:generalthm}, we have \(\mathfrak J(\mathcal B)>0\) and \(\mathfrak S(u,v)>0\). Therefore, the leading constant is strictly positive, completing the proof of Theorem~\eqref{eq:generalthm}.

The requirement \(s\ge10\) arises in the minor-arc argument: one factor is controlled by the pointwise estimate \eqref{eq:minorpointwise}, the other nine are controlled by the ninth-moment estimate \eqref{eq:momentinput}, and any remaining factors are bounded trivially.

\section{Positive points on the homogeneous cone}\label{sec:cone}

We next establish the geometric input for bordering. These lemmas are
stated for an arbitrary number $s\geq10$ of active coordinates. Throughout this section, let \(m\ge s\) and
\(a_j,b_j\in\mathbb Z\setminus\{0\}\) for \(1\le j\le s\).
Write
\begin{equation}\label{eq:coneforms}
 L_0(\by)=\sum_{j=1}^s a_jy_j,\qquad
 C_0(\by)=\sum_{j=1}^s b_jy_j^3,\qquad a_jb_j\ne0,
\end{equation}
and
\[
 X_m=\{\by\in\CC^m:L_0(\by)=C_0(\by)=0\}.
\]
The last $m-s$ coordinates are free.

\subsection{Absolute irreducibility by polarization}

\begin{lemma}\label{lem:isotropic}
Let $b_1,\ldots,b_s\in\CC\setminus\{0\}$ and
$D(\by)=\sum_{j=1}^s b_jy_j^3$.
If a complex linear subspace $W\subseteq\CC^s$ satisfies $D|_W=0$,
then $\dim W\leq s/2$.
\end{lemma}

\begin{proof}
Let $T$ be the set of coordinates that do not vanish identically on $W$. The case when $W = 0$ is immediate.
Choose $\bw\in W$ with $w_j\ne0$ for every $j\in T$; this is possible
because a finite union of proper linear subspaces does not cover a vector
space over an infinite field. For \(\bu,\bv\in W\), polarization gives
\begin{align*}
 D(\bw+\bu+\bv)-D(\bw+\bu)-D(\bw+\bv)-D(\bu+\bv)
 +D(\bw)+D(\bu)+D(\bv)
   =6\sum_{j=1}^s b_jw_ju_jv_j.
\end{align*}
All seven arguments on the left lie in $W$, so $D|_W=0$ gives
\[
 \sum_{j\in T}b_jw_ju_jv_j=0\qquad \text{for }\bu,\bv\in W.
\]

Every coordinate outside \(T\) vanishes on \(W\), so we may
identify \(W\) with a subspace of \(\mathbb C^T\).
On $\mathbb{C}^T$, the bilinear form
\[
B_{\bw}(\bu,\bv)
=\sum_{j\in T}b_jw_ju_jv_j
\]
is not degenerate, since its matrix in the standard basis of $\mathbb{C}^T$ is diagonal with nonzero entries $b_jw_j$ for $j \in T$.    
The preceding identity gives \(W\subseteq W^\perp\),
where the orthogonal complement is taken with respect to
\(B_{\mathbf w}\). Hence
\[
\dim W\le\dim W^\perp=|T|-\dim W,
\]
and we have \(2\dim W\le |T|\le s\).
\end{proof}

\begin{lemma}\label{lem:irreducible}
If $s\geq5$, the restriction of $C_0$ to the hyperplane $L_0=0$ in
$\CC^s$ is a nonzero absolutely irreducible cubic polynomial.
Consequently, $X_m$ is absolutely irreducible.
\end{lemma}

\begin{proof}
Let $H=V(L_0) \subseteq \mathbb{C}^s$. Since $\dim H=s-1>s/2$, Lemma~\ref{lem:isotropic} implies $C_0|_H\ne0$.
Suppose $C_0|_H$ is reducible. Being homogeneous of degree
three, it has a homogeneous factor of degree one, say $\ell$.
Then $C_0$ vanishes on the linear subspace
$\ker\ell\subseteq H$, whose dimension is $s-2>s/2$
for $s\ge5$, contradicting Lemma~\ref{lem:isotropic}.

Let $X_s=V(L_0,C_0)\subseteq H$. Since $C_0|_H$ is irreducible and $H\cong\mathbb C^{s-1}$, the ideal $(C_0|_H)$ is prime, so
\[
\mathbb C[X_s]\cong\mathbb C[H]/(C_0|_H)
\]
is an integral domain. The remaining $m-s$ coordinates are free, so
\[
\mathbb C[X_m]\cong\mathbb C[X_s][t_1,\ldots,t_{m-s}]
\]
is also an integral domain. Hence $X_m$ is irreducible over
$\mathbb C$, proving absolute irreducibility.
\end{proof}

We shall also use the following elementary form of real Zariski density.

\begin{lemma}\label{lem:realpatch}
Let $F\in\R[z_1,\ldots,z_N]$ be absolutely irreducible, and let $U$ be a nonempty relatively Euclidean open subset of the nonsingular real locus of $F = 0$. If \(G\in\mathbb R[z_1,\ldots,z_N]\) has a nonzero class modulo \(F\), then \(G\) does not vanish throughout \(U\).
\end{lemma}

\begin{proof}
Suppose that \(G\) vanishes on \(U\). After restricting
to a smaller patch and relabelling the coordinates,
we may assume \(\partial F/\partial z_N\ne0\) throughout
the patch. The implicit function theorem expresses it
as a graph over a nonempty open subset of
\(\mathbb R^{N-1}\).

The polynomial \(F\) is irreducible over \(\mathbb R\)
and has positive degree in \(z_N\). It is therefore
primitive as a polynomial in \(z_N\) over
\(\mathbb R[z_1,\ldots,z_{N-1}]\).
Since \(F\) does not divide \(G\), Gauss's lemma shows
that they are coprime in
\(\mathbb R(z_1,\ldots,z_{N-1})[z_N]\).
The Euclidean algorithm, followed by clearing denominators,
gives polynomials \(A_0,B_0\in\mathbb R[z_1,\ldots,z_N]\)
and a nonzero polynomial
\(h\in\mathbb R[z_1,\ldots,z_{N-1}]\) such that
\[
A_0F+B_0G=h.
\]
Both \(F\) and \(G\) vanish on the graph, so \(h\)
vanishes on a nonempty open subset of \(\mathbb R^{N-1}\), which contradicts $h\ne 0$.
\end{proof}

\subsection{A nonsingular point among the positive rows}

\begin{lemma}\label{lem:positiverow}
Let $E$ be a positive real $(m-1)\times m$ matrix of rank $m-1$. Suppose that the row spaces of $E$ and $\cube E$ are respectively the hyperplanes defined by $L_0=0$ and
$\sum_{j=1}^s b_jy_j=0$, with the $b_j$'s as in
\eqref{eq:coneforms}. If $s\geq3$, at least one row of $E$ is a
nonsingular point of $X_m$.
\end{lemma}

\begin{proof}
Let $y$ be a singular row of $E$. Then $\nabla C_0(\by)$ and
$\nabla L_0(\by)$ are linearly dependent. Since both are nonzero,
there is a nonzero real number \(\lambda\) such that
\begin{equation}\label{eq:singularrow}
3b_jy_j^2=\lambda a_j
\qquad \text{for } 1\le j\le s.
\end{equation}
For two singular rows $y,\tilde y$, with corresponding
parameters $\lambda,\tilde\lambda$, we have
\[
\frac{y_j^2}{\tilde y_j^2}
=\frac{\lambda}{\tilde\lambda}>0,
\qquad 1\le j\le s.
\]
Since the rows are positive, taking square roots shows that
their first $s$ coordinates are proportional. Thus if every row of $E$ were singular, the submatrix consisting of the
first $s$ columns of $E$ would have rank at most one. Since
the remaining $m-s$ columns contribute at most $m-s$ to the rank,
\[
m-1=\operatorname{rank}E\le1+(m-s)\le m-2,
\]
where the last inequality uses $s\ge3$. This is a contradiction.
\end{proof}

\begin{lemma}\label{lem:conecount}
Assume $s\geq10$ and that $X_m$ contains a positive integral point
where the Jacobian of $(L_0,C_0)$ has rank two.
Let \(\mathcal U=\prod_{j=1}^{m}I_j\) be any positive product box containing a nonsingular
real point of $X_m$. Then
\[
 \#\{\by\in\Z^m:\by/P\in\mathcal U,\ L_0(\by)=C_0(\by)=0\}
 =\kappa_{\mathcal U}P^{m-4}
       +O_{\mathcal U}(P^{m-4-1/100}),
 \qquad \kappa_{\mathcal U}>0.
\]
\end{lemma}

\begin{proof}
The given integral point is a solution of the homogeneous equations over every $\Z_p$. Some Jacobian minor is a nonzero integer, hence is nonzero over every $\Q_p$.

Apply Theorem~\ref{thm:analytic} to the first \(s\) coordinates, with targets \(0,0\) and box \(\prod_{j=1}^{s}I_j\). The local condition follows from the integral point above. The real condition follows from the nonsingular point in \(\mathcal U\), since both forms depend only on the first \(s\) coordinates. Thus, the count in these
coordinates is
$\kappa P^{s-4} +O_{\mathcal U}(P^{s-4-1/100}) \text{ for } \kappa>0.$
If \(m>s\), the remaining coordinates contribute
\[
\prod_{j=s+1}^{m}\#(\mathbb Z\cap PI_j)
=
P^{m-s}\prod_{j=s+1}^{m}|I_j|
+O_{\mathcal U}(P^{m-s-1}).
\]
Multiplying these counts gives the result, with
\(\kappa_{\mathcal U}
=\kappa\prod_{j=s+1}^{m}|I_j|>0\).
If \(m=s\), there are no extra coordinates and
\(\kappa_{\mathcal U}=\kappa\).
\end{proof}

\subsection{Simultaneous row relations with nonzero coefficients}

\begin{lemma}\label{lem:sharedrow}
Let $m\geq s\geq10$. Let $E$ be a positive integer $(m-1)\times m$ matrix for which both signed cofactor vectors have the same support $\{1,\ldots,s\}$. For every real $K$, there are an integral row $\br$ with all $r_j>K$ and $r_j>1$ for every $j$, and integral rows $\bu,\bv$ of length $m-1$, such that
\begin{equation}\label{eq:sharedrelations}
 \br=\bu E,\qquad \br^{[3]}=\bv E^{[3]},\qquad
 u_iv_i\ne0\quad \text{for }1\leq i\leq m-1.
\end{equation}
\end{lemma}

\begin{proof}

Write the signed cofactor vectors of $E$ and $E^{[3]}$ as
$(a_1,\ldots,a_s,0,\ldots,0)$ and
$(b_1,\ldots,b_s,0,\ldots,0)$, respectively.
Both matrices have rank $m-1$. Hence, for each $\by\in X_m$,
there are unique coefficient vectors $\bu(\by),\bv(\by)$ satisfying
\begin{equation}\label{eq:relationfunctions}
 \by=\bu(\by)E,\qquad \by^{[3]}=\bv(\by)E^{[3]}.
\end{equation}
Let $J=\{2,\ldots,m\}$, and let $E_J$ denote the submatrix
formed by these columns. The first cofactor of each matrix
is nonzero, so
\begin{equation}\label{eq:relationformulas}
 \bu(\by)=\by_J E_J^{-1},\qquad
 \bv(\by)=\by_J^{[3]}(E_J^{[3]})^{-1}.
\end{equation}
Thus the coordinates $u_i$ and $v_i$ are homogeneous
polynomials with rational coefficients, of degrees one
and three, respectively.

For the $i$-th row $E_i$ of $E$, we have
$\bu(E_i)=\bv(E_i)=\boldsymbol{e}_i$, so $u_i(E_i)=v_i(E_i)=1$.
Therefore none of the functions $u_i,v_i$ vanishes
identically on $X_m$. By Lemma~\ref{lem:irreducible}, $\mathbb C[X_m]$
is an integral domain, so
\[
G=\prod_{i=1}^{m-1}u_iv_i
\]
does not vanish identically on $X_m$ either.

By Lemma~\ref{lem:positiverow}, $X_m$ contains a positive nonsingular row
of $E$. Since $H=V(L_0)\cong\mathbb C^{m-1}$, we may
regard $X_m$ as the absolutely irreducible cubic
hypersurface in $H$ defined by $C_0|_H$.
Lemma~\ref{lem:realpatch} therefore gives a positive nonsingular real point $\bx\in X_m$ with $G(\bx)\ne0$.
Choose a bounded open box $\mathcal U$ containing $\bx$ with $\overline{\mathcal U}\subset(0,\infty)^m$, on which $G$ is nowhere zero.

By Lemma~\ref{lem:conecount}, for every sufficiently large $P$, there is an integral point $\by\in X_m$ with $\by/P\in\mathcal U$.
Homogeneity gives
\[
G(\by)=P^{4(m-1)}G(\by/P)\ne0,
\]
so every $u_i(\by)$ and $v_i(\by)$ is nonzero.
Set
\[
M=\left|\det E_J\,\det E_J^{[3]}\right|,
\qquad
\br=M\by,\qquad \bu=M\bu(\by),\qquad \bv=M^3\bv(\by).
\]
The adjugate formulas show that $u$ and $v$ are integral:
$M$ is divisible by $\det E_J$, and $M^3$ is divisible
by $\det E_J^{[3]}$. Furthermore,
\[
\br=\bu E,\qquad \br^{[3]}=\bv E^{[3]},
\qquad u_iv_i\ne0\quad \text{for }1\le i\le m-1.
\]
Taking $P$ sufficiently large ensures
$r_j>\max\{K,1\}$ for every $j$, as required.
\end{proof}

\section{A cofactor-preserving bordering construction}\label{sec:border}

The next lemma gives the cofactor identities used in the
construction. Vectors multiplying a matrix on the left
are row vectors.

\begin{lemma}\label{lem:borderidentity}
Let $E$ have size $(m-1)\times m$, and suppose
$\br=\bu E$ and $\br^{[3]}=\bv E^{[3]}$.
For a column $\bw=(w_1,\ldots,w_{m-1})^{\mathsf T}$ and a scalar $z$,
put
\begin{equation}\label{eq:borderblock}
 E^+=\begin{pmatrix}E&\bw\\\br&z\end{pmatrix},\qquad
 k_1=z-\bu\bw,\qquad k_3=z^3-\bv\bw^{[3]}.
\end{equation}
Then
\begin{equation}\label{eq:bordercofactor}
 \bgamma(E^+)=(k_1\bgamma(E),0),\qquad
 \bgamma((E^+)^{[3]})=(k_3\bgamma(E^{[3]}),0).
\end{equation}
\end{lemma}

\begin{proof}
For \(j\le m\), delete column \(j\) from \(E^+\).
Subtract the linear combination of the first \(m-1\)
rows with coefficients \({\bu}\) from the last row.
Since \({\br}_{\widehat j}={\bu}E_{\widehat j}\),
the last row becomes \((0,\ldots,0,k_1)\).
Expansion along this row gives determinant
\(k_1\det E_{\widehat j}\), with positive expansion sign.
The signed-cofactor factor \((-1)^{1+j}\) is the same
for \(E\) and \(E^+\).

Deleting the last column gives
\(\begin{pmatrix}E\\{\br}\end{pmatrix}\),
whose rows are dependent, so the last cofactor is zero.
The same argument applied to \((E^+)^{[3]}\), using
$\br^{[3]}=\bv E^{[3]}$, gives the second identity.
\end{proof}

\begin{remark}
In general, $(\bu E)^{[3]}\ne \bu E^{[3]}$, so the relation \({\br}={\bu}E\) alone does not
give the second identity in \eqref{eq:bordercofactor}. The cubed cofactors require the separate relation
\({\br}^{[3]}={\bv}E^{[3]}\).
\end{remark}

\begin{lemma}\label{lem:auxiliary}
Let \(E\) be a positive real \((m-1)\times m\) matrix
of rank \(m-1\), and let \({\br}\) be positive.
Suppose \eqref{eq:sharedrelations} holds with integral $\bu,\bv$ whose
coordinates are all nonzero. If $m\geq10$, then, for every real $K$,
there are integers $w_1,\ldots,w_{m-1},z$, all greater than $K$ and one,
such that
\begin{equation}\label{eq:unitcomplements}
 z-\sum_{i=1}^{m-1}u_iw_i=1,\qquad
 z^3-\sum_{i=1}^{m-1}v_iw_i^3=1.
\end{equation}
\end{lemma}

\begin{proof}
The system consists of a linear equation and a diagonal
cubic in \(m\) variables, with coefficient vectors
\begin{equation}\label{eq:auxcoeff}
 (-u_1,\ldots,-u_{m-1},1),\qquad
 (-v_1,\ldots,-v_{m-1},1).
\end{equation}
All coefficients are nonzero. We check the local and
real conditions of Theorem~\ref{thm:analytic}.

For every prime $p$, the point
\begin{equation}\label{eq:trivialaffinepoint}
 (w_1,\ldots,w_{m-1},z)=(0,\ldots,0,1)
\end{equation}
is an integral solution of \eqref{eq:unitcomplements}.
The Jacobian minor in columns $w_i,z$ equals $-3u_i$ at this point. Since \(u_i\ne0\), this minor
is nonzero over every \(\mathbb Q_p\), giving the required
nonsingular \(\mathbb Z_p\)-solution.

For the homogeneous real equations, take any column \(j\)
of \(E\) and set \(w_i=E_{ij}\) and \(z=r_j\).
This point is positive, and the two equations follow
from \({\br}={\bu}E\) and
\({\br}^{[3]}={\bv}E^{[3]}\).

At least one of these points is nonsingular.
Otherwise, the Jacobian minor
\[
\det\begin{pmatrix}
-u_i&1\\
-3v_iw_i^2&3z^2
\end{pmatrix}
=3(v_iw_i^2-u_iz^2)
\]
would vanish at every column point, giving
\begin{equation}\label{eq:columnsingularity}
 v_iE_{ij}^{\,2}=u_i r_j^2
 \qquad\text{for }1\leq i\leq m-1,\ 1\leq j\leq m.
\end{equation}
Since \(E_{ij},r_j>0\), this implies \(u_i/v_i>0\) and
\[
E_{ij}=\sqrt{\frac{u_i}{v_i}}\,r_j
\qquad(1\le j\le m).
\]
For each \(i\), the multiplier is independent of \(j\).
Thus every row of \(E\) is proportional to \({\br}\),
so \(\operatorname{rank}E\le1\), contradicting
\(\operatorname{rank}E=m-1\ge2\).

Choose a sufficiently small positive product box
$\mathcal B$ around a nonsingular column point.
By Theorem~\ref{thm:analytic}, applied with
$s=m$ and $u=v=1$, we have
\[
R(P)=\kappa P^{m-4}
+O(P^{m-4-1/100}),
\qquad \kappa>0.
\]
Thus $R(P)>0$ for all sufficiently large $P$.
Since every coordinate interval of $\mathcal B$
is bounded away from zero, the resulting
integer solutions have all coordinates
greater than $K$ and $1$ for sufficiently large $P$.
\end{proof}

\begin{theorem}[One-step extension]\label{thm:border}
Let $m\geq10$ and let $E$ be an integer $(m-1)\times m$ matrix whose
entries are at least two. Suppose that its signed cofactor vector and
that of $E^{[3]}$ have the same support, of cardinality $s\geq10$.
For every real $K$, there is an integer $m\times(m+1)$ matrix $E^+$,
with $E$ as its upper-left block, such that every new entry is greater
than $K$ and one, and
\begin{equation}\label{eq:unitborder}
 \bgamma(E^+)=(\bgamma(E),0),\qquad
 \bgamma((E^+)^{[3]})=(\bgamma(E^{[3]}),0).
\end{equation}
\end{theorem}

\begin{proof}
Temporarily permute the columns of \(E\) so that the
common cofactor support is \(\{1,\ldots,s\}\).
Column permutations permute the signed cofactor vectors
up to a common sign, so the hypotheses of Lemma~\ref{lem:sharedrow}
still hold. Apply that lemma to obtain
\({\br},{\bu},{\bv}\) satisfying \eqref{eq:sharedrelations}, with every
entry of \({\br}\) greater than \(K\) and one.
Restore the original column order in both \(E\) and
\({\br}\); the two row relations remain valid with
the same \({\bu},{\bv}\).

The nonzero cofactors imply \(\operatorname{rank}E=m-1\).
Lemma~\ref{lem:auxiliary} therefore gives integral \({\bw},z\) with the required entry bounds and \(k_1=k_3=1\), form \(E^+\) as in \eqref{eq:borderblock}. Lemma~\ref{lem:borderidentity} gives \eqref{eq:unitborder}.
\end{proof}

\begin{remark}\label{rem:notduplicate}
Taking $\br=tE_i$ for a single old row does not give this construction. In that case, $\bu=t\boldsymbol e_i$ and $\bv=t^3\boldsymbol e_i$, so the two equations
with \(k_1=k_3=1\) become
\[
 z-tw_i=1,\qquad z^3-(tw_i)^3=1.
\]
Substituting \(z=tw_i+1\) into the second equation gives
\(3(tw_i)(tw_i+1)=0\), which is impossible for \(t,w_i>0\).
The simultaneous relations with all coefficients nonzero, rather than a duplicated row, are essential to the construction.
\end{remark}

\section{Proofs of the main results and consequences}\label{sec:finish}

\begin{proof}[Proof of Theorem~\ref{thm:seed}]
Apply Theorem~\ref{thm:analytic} with the coefficients in
\eqref{eq:cofdef}, the box \eqref{eq:box}, \(s=10\) and \(u=v=1\). All coefficients are nonzero by the table in Section~\ref{sec:cofactor}. 
The vector $\br$ defined in \eqref{eq:box}
lies in $\cB$, satisfies $L(\br)=C(\br)=0$
by \eqref{eq:realzero}, and has Jacobian
rank two by \eqref{eq:realminor}.
Proposition~\ref{prop:alllocal} verifies integral nonsingular local solubility at every prime. Thus
\[
\kappa=\fS(1,1)\fJ(\cB)>0.
\]
Equation \eqref{eq:LC} identifies the counted solutions
with the determinant conditions, giving \eqref{eq:maincount}.

Every entry of \(B\) is at least two, and \(\bx/P\in\cB\)
implies \(x_j>1.9P\) for every \(j\). For all sufficiently
large \(P\), the asymptotic count is positive and all
first-row entries exceed one. Choosing solutions along
a sequence \(P\to\infty\) makes every first-row entry
tend to infinity, giving infinitely many distinct matrices.
\end{proof}

\begin{proof}[Proof of Theorem~\ref{thm:frames}]
Set $B_{10}=B$. Its entries are at least two, and all coordinates of \(\bc\) and \(\bd\) are nonzero by Section~\ref{sec:cofactor}. Suppose \(B_m\) has been constructed with the required properties. Apply Theorem~\ref{thm:border} to \(E=B_m\) with \(K=1\), and set \(B_{m+1}=E^+\). This retains \(B_m\) as the upper-left block and appends a zero to each cofactor vector. The new entries are integers greater than one, and the common cofactor support still has size ten. Induction gives \eqref{eq:stablecofactors} and the
nesting property for every \(n\ge10\).

Laplace expansion along the first row gives
\[
 \det\begin{pmatrix}\by\\B_n\end{pmatrix}
   =\sum_{j=1}^n\gamma_j(B_n)y_j
   =\sum_{j=1}^{10}c_jy_j,
\]
The same expansion for the entrywise cube gives
$\sum_{j=1}^{10}d_jy_j^3$, which proves
\eqref{eq:stableequations}.
\end{proof}


\begin{proof}[Proof of Theorem~\ref{thm:main}]
Choose a first row \(\bx\) supplied by
Theorem~\ref{thm:seed}, with every entry at least two. For each \(n\ge10\), form
\begin{equation}\label{eq:nestedA}
 A_n=\begin{pmatrix}(\bx,2,\ldots,2)\\ B_n\end{pmatrix},
\end{equation}
where the lower block comprises all $n-1$ rows of $B_n$.
Its two determinants equal one by \eqref{eq:stableequations}. Since \(B_n\)
is the upper-left block of \(B_{n+1}\) and the first
\(n\) entries of the top row are unchanged, \(A_n\)
is the upper-left \(n\times n\) block of \(A_{n+1}\).
These compatible matrices define the required infinite array.

Theorem~\ref{thm:seed} supplies infinitely many distinct
choices of \(\bx\) with entries at least two. Distinct
choices remain distinct in \(A_n\), so there are
infinitely many matrices at every fixed order \(n\ge10\).
\end{proof}

\begin{proof}[Proof of Corollary~\ref{cor:allcount}]
The determinant equations in \eqref{eq:stableequations}
involve only the first ten variables. Each remaining
coordinate is any integer in \((2P,3P)\), so the count
is exactly the order-ten count multiplied by
\[
\bigl(\#\{t\in\Z:2P<t<3P\}\bigr)^{n-10}
=P^{n-10}\bigl(1+O_n(P^{-1})\bigr).
\]
For \(n=10\), this factor is one. Multiplying by
\eqref{eq:maincount} gives \eqref{eq:allcount}, with
the same leading constant \(\kappa\); the additional
error is \(O_n(P^{n-5})\).
\end{proof}

\subsection{Large lower bounds in every sufficiently large order}

\begin{lemma}\label{lem:extendmatrix}
Let $A$ be an integer matrix of order $m\geq10$ whose entries are
at least two, with $\det A=\det A^{[3]}=1$. Suppose the signed cofactors along its first
row and those along the first row of $A^{[3]}$ have common support of
size at least ten. Then $A$ is the upper-left block of a matrix of order
$m+1$ with the same two determinants and with every new entry exceeding
any prescribed bound. The same cofactor-support property holds after
extension.
\end{lemma}

\begin{proof}
Let \(E\) be the lower \(m-1\) rows of \(A\), and let
\(K\) be the prescribed bound. Its two cofactor vectors
are precisely the first-row cofactor vectors of \(A\)
and \(A^{[3]}\). Apply Theorem~\ref{thm:border} to obtain
\(E^+\), with all new entries greater than \(K\) and one.
Append an integer \(q>\max\{K,1\}\) to the first row
of \(A\), and use \(E^+\) as the lower block of the
enlarged square matrix.

By \eqref{eq:unitborder}, the first-row cofactor vectors
are obtained by appending zero to the old vectors,
both before and after entrywise cubing. Laplace expansion
therefore gives the original determinants, both equal
to one. The common cofactor support is unchanged.
\end{proof}

\begin{proof}[Proof of Corollary~\ref{cor:alllarge}]
Choose an order-ten matrix \(A\) from
Theorem~\ref{thm:seed} with every entry at least two.
Its first-row cofactor vectors are \(\bc\) and \(\bd\),
with every coordinate nonzero. For a fixed integer
\(r\ge1\), let \(T=A^{\otimes r}\), of order \(N=10^r\).
The identities
\[
(A\otimes H)^{[3]}=A^{[3]}\otimes H^{[3]},
\qquad
\det(A\otimes H)
=(\det A)^{\operatorname{ord}H}
 (\det H)^{\operatorname{ord}A}
\]
show that \(\det T=\det T^{[3]}=1\). Each entry of \(T\)
is a product of \(r\) entries of \(A\), so it is at
least \(2^r\).

For a determinant-one matrix, the transpose of its
first-row cofactor vector is the first column of its
inverse. Since \(T^{-1}=(A^{-1})^{\otimes r}\), its
first column is \((\bc^\top)^{\otimes r}\). Applying
the same argument to \(T^{[3]}=(A^{[3]})^{\otimes r}\)
shows that the first-row cofactor vectors of \(T\)
and \(T^{[3]}\) are \(\bc^{\otimes r}\) and
\(\bd^{\otimes r}\), respectively. Every coordinate
of both vectors is nonzero, so
Lemma~\ref{lem:extendmatrix} applies.

Starting from \(T\), apply this lemma repeatedly with
prescribed bound \(2^r-1\). Every new entry is then an
integer at least \(2^r\), and the existing entries are
unchanged. This gives the required matrices in every
order \(n\ge10^r\). For the final assertion, choose
\(r\ge1\) with \(2^r>K\).
\end{proof}

\subsection{Nonexistence in orders one and two}

\begin{proposition}\label{prop:small}
There is no integer matrix of order one or two with both determinants
equal to one and every entry of absolute value greater than one.
\end{proposition}

\begin{proof}
In order one, the determinant condition forces the only
entry to equal one. For order two, write \(U=a_{11}a_{22}\)
and \(V=a_{12}a_{21}\). The determinant equations give
\[
U-V=1,\qquad U^3-V^3=1.
\]
Factoring the second equation and using the first gives
\(U^2+UV+V^2=1\). Substituting \(U=V+1\) yields
\(3V(V+1)=0\), so \((U,V)=(1,0)\) or \((0,-1)\).
Neither is possible when \(|a_{ij}|>1\), which would force $|U|,|V|\geq4$.
\end{proof}

\section{Further remarks}\label{sec:scope}

Theorem~\ref{thm:main} shows there exists infinitely many positive solutions in \emph{every} order $n\ge10$, while Proposition \eqref{prop:small} rules out orders one and two. Our construction leaves open the existence of positive solutions in orders $3$ through $9$. Another natural question is whether, for a fixed order $n\ge10$, the minimum entry can be arbitrarily large. Corollary \ref{cor:alllarge} establishes that for any  integer $r\ge1$, there exist matrices of every order $n\ge10^r$ whose entries are all at least $2^r$.

Our proofs also yield a theoretically terminating exhaustive search procedure. All searches are performed by enumerating integer
vectors in increasing order of their maximum
absolute coordinate. To find a valid first row for the order-10 case, one may enumerate positive integer rows (with entries taken at least two for this search scheme) and check the two determinant equations; Theorem \ref{thm:seed} ensures that this search must terminate.
At each extension step, we enumerate integral rows with entries at least two satisfying the homogeneous cofactor equations until all relation coefficients in \eqref{eq:relationformulas} are nonzero, then we  clear the denominators of the corresponding rational row-relation coefficients as in Lemma~\ref{lem:sharedrow}. After that, we enumerate the integer pairs $(\boldsymbol{w},z)$ that meet the required entry bounds and verify the equations in \eqref{eq:unitcomplements}. Lemmas \ref{lem:sharedrow} and \ref{lem:auxiliary} guarantee termination of these searches, although we do not provide any bound on the running time.

The seed block, its cofactor vectors, and the small-prime certificates are all explicit. The accompanying program verifies these finite arithmetic data and tests the cofactor identity underlying the bordering construction. Details of the exact-arithmetic checks are given in Appendix A.

\appendix
\section{Verification of the finite arithmetic certificates}\label{app:verify}

Our paper has an accompanying file \texttt{verify\_all\_orders.py}, which is a self-contained Python 3.9+ program using
only the standard library. It can be run with
\begin{lstlisting}
python verify_all_orders.py
\end{lstlisting}
This program contains the seed block \eqref{eq:B}, both cofactor vectors and the five local certificates from Section 4.2. All calculations use arbitrary-precision integers and
exact rational numbers.

The twenty signed cofactors of \(B\) and \(B^{[3]}\) are computed
from their \(9\times9\) minors. Each determinant is evaluated by
both fraction-free Bareiss elimination and rational Gaussian
elimination with row pivoting. The results are compared before
the cofactor sign \((-1)^{1+j}\) is applied. At a pivot \(a_{kk}\),
the Bareiss update is
\[
a_{ij}\leftarrow
\frac{a_{kk}a_{ij}-a_{ik}a_{kj}}
     {\text{previous pivot}}
\qquad \text{for }i,j>k,
\]
with the previous pivot initially set to one.

Using these cofactors,  the program verifies the two primitive-vector identities and the Bézout identity~ \eqref{eq:bezout}. It recomputes and prints all $45$ pairwise gcds, checks that each divides \(840\), and verifies that their least common multiple is \(840\). It also verifies \(L(\mathbf r)=C(\mathbf r)=0\) and the exact value in~\eqref{eq:realminor}. The repeated-row identities are checked for every row of \(B\).

For each small-prime certificate, the program evaluates
\(L(\mathbf x)\bmod m\), \(C(\mathbf x)\bmod m\), and
\(D_{ij}(\mathbf x)\bmod p\). In the \(3\)-adic case,
eliminating \(x_3\) gives an exact rational residual and its
derivative; the program verifies that the residual is divisible
by \(27\) and that the derivative has \(3\)-adic valuation one.
The estimates used for the large primes are checked through
the exact inequalities
\[
13^5>600^2,\qquad
512\cdot1000<13^6,\qquad
\frac{64}{75}+\frac{1}{1000}<1.
\]

A successful default run will include the message
\begin{lstlisting}
ALL FINITE CERTIFICATES AND ALGEBRAIC REGRESSION TESTS PASSED.
\end{lstlisting}
The default run checks the seed certificates and the algebraic bordering identities. 

\subsection{The bordering certificate checker}\hfill

The function \path{check_border_certificate(E,r,u,v,w,z)} checks both row identities in \eqref{eq:sharedrelations}, the two
integer complements, and both cofactor vectors in \eqref{eq:bordercofactor}.
All cofactors are computed directly using the two determinant
algorithms above. In strict mode, it requires every
entry be at least two, all relation coefficients be
nonzero, and both complements equal to one.

The default run tests \eqref{eq:bordercofactor} on duplicate-row data, including a column permutation and an iterated extension. These tests check the general cofactor identity; the test data do not satisfy
the requirements for admissible unit-complement borders.

In particular, the program does not construct an admissible first row in order ten or an admissible order-eleven extension. The existence of these objects follows from the analytic and geometric arguments in the our main text.


\end{document}